\documentclass[12pt,reqno]{amsart}
  \usepackage{latexsym}
  \usepackage{mathtools}
  \usepackage[all]{xy}
  \usepackage{amsfonts} 
  \usepackage{amsthm} 
  \usepackage{amsmath} 
  \usepackage{amssymb}
  \usepackage{pifont}  
  \usepackage{enumerate}
  \usepackage{dcolumn}
  \usepackage{comment}
  \usepackage{hyperref}
  \usepackage{enumitem}
  \usepackage{tikz}
   \usepackage[capitalise]{cleveref}
\usetikzlibrary{arrows}
\newcolumntype{2}{D{.}{}{2.0}}
  \xyoption{2cell}

  \def\Label#1{\label{#1}\ifmmode\llap{[#1] }\else 
  \marginpar{\smash{\hbox{\tiny [#1]}}}\fi} 
  \def\Label{\label} 

  \newtheorem{proposition}{Proposition}[section]
  \newtheorem{lemma}[proposition]{Lemma} 
  \newtheorem{corollary}[proposition]{Corollary} 
  \newtheorem{theorem}[proposition]{Theorem} 
  
\theoremstyle{definition} 
  \newtheorem{definition}[proposition]{Definition}
    
  \newtheorem{example}[proposition]{Example}

  \theoremstyle{remark} 
  \newtheorem{remark}[proposition]{Remark}

 \newcounter{zlist}

  \newcounter{blist}

  \newcounter{rlist}

  \newcounter{c} 
  \renewcommand{\[}{\setcounter{c}{1}$$} 
  \newcommand{\etyk}[1]{\vspace{-7.4mm}$$\begin{equation}\Label{#1} 
  \addtocounter{c}{1}} 
  \renewcommand{\]}{\ifnum \value{c}=1 $$\else \end{equation}\fi} 
   \numberwithin{equation}{section}

\def\NN{{\mathbb N}}

\def\ZZ{{\mathbb Z}}

\newcommand{\bB}{\mathrm{B}}

\newcommand{\hH}{\mathrm{H}}

\newcommand{\lL}{\mathrm{L}}

\newcommand{\rR}{\mathrm{R}}

\newcommand{\tT}{\mathrm{T}}

\newcommand{\zZ}{\mathrm{Z}}

\def\ll{\langle}
\def\rr{\rangle}

\newcommand{\im}{{\rm Im}\,}

   \def\<{{\langle}} 
   \def\>{{\rangle}}

   \def\text#1{{\rm {\rm #1}}}

   \def\|#1{\overline{#1}}
   
   \def\h#1 {\hat{#1}}

\title{Boolean trusses as rectangular bands and Boolean rings}
\author{Ryszard R. Andruszkiewicz}
\address{
Faculty of Mathematics, University of Bia{\l}ystok, K.\ Cio{\l}kowskiego  1M,
15-245 Bia\-{\l}ys\-tok, Poland}

\email{r.andruszkiewicz@uwb.edu.pl}
\author{Bernard Rybo\l owicz}
\address{Independent Researcher}\email{bernard.rybolowicz@icloud.com}



\begin{document}

\begin{abstract}
We introduce Boolean trusses, that is, trusses in which every element is idempotent. Every Boolean truss admits a decomposition into a truss arising from a Boolean ring and a rectangular truss. Furthermore, we classify all Boolean trusses up to isomorphism.
\end{abstract}  
\maketitle

\section{Introduction}

It is well known that Boolean rings and Boolean algebras are closely related and can be treated as equivalent algebraic structures. A classical result due to Stone states that every Boolean ring is isomorphic to a ring of sets \cite{Stone1936}. In other words, every Boolean ring admits a representation as a concrete collection of subsets of a suitable set, closed under symmetric difference and intersection. However, such a representation need not be the full power set of a set. For example, the Boolean ring of all finite subsets of $\mathbb{N}$ is countable, whereas the power set of any infinite set is uncountable. For further background on Boolean algebras and Stone duality, we refer the reader to \cite{Johnstone}.

If the additive group structure of a Boolean ring is forgotten, the remaining algebraic structure is an idempotent semigroup called a band, and, in particular, it is commutative. In contrast to Boolean rings, bands need not be commutative. In 1954, Clifford \cite{Clifford1954} extended a result of McLean \cite{McLean1954} and proved that every band is a semilattice of rectangular bands, that is, bands that are direct products of two semigroups with operations $ab=a$ and $ab=b$. Although rectangular bands are classical objects in semigroup theory, they continue to appear in modern research. See, for example, \cite{Cameron,Khrypchenko}.

Recently, in \cite{Brzezinski2019}, T. Brzezi\'nski introduced a new algebraic structure, a truss. Trusses were introduced as a bridge between braces, which were introduced in \cite{Rump2007} to study solutions of set-theoretic Yang-Baxter equations, and rings. By replacing the additive group binary operation with a ternary Abelian heap operation in the axioms of a ring, one obtains a truss. A two-sided brace is a special case of a truss. A truss is a generalisation of a ring and can be associated with particular congruence classes of a ring; see \cite{ABR}.

Following the idea of Boolean rings, i.e. rings in which every element is an idempotent, we study Boolean trusses, that is, trusses in which every element is an idempotent. Every Boolean ring is commutative, while for Boolean trusses this is not true in general. On the other hand, in semigroup theory there are noncommutative semigroups that consist only of idempotents, for example, the so-called left zero semigroup $L$, where $ab=a$ for all $a,b\in L$, and the right zero semigroup $R$, where $ab=b$ for all $a,b\in R$. These two examples of semigroup operations do not form a ring with any Abelian group, but they form a truss with every Abelian heap. Such trusses we will call left zero and right zero trusses, respectively.

We show that Boolean rings, left zero semigroups, and right zero semigroups contribute to a description of Boolean trusses. That allows us to fit them into ring theory as particular congruence classes. The problem of describing all Boolean trusses arose during the Bia\l ystok University Algebra Seminar. This paper addresses this problem and describes all Boolean trusses up to isomorphism.

The article is organised as follows:

In the preliminaries, we recall all the necessary facts about bands, heaps, and trusses. In the main part, we introduce Boolean trusses as trusses in which every element is idempotent and rectangular trusses as heaps with a compatible rectangular band structure. In the main theorem, Theorem \ref{thm:main}, we show that every Boolean truss $T$ decomposes into a product of four trusses:
$$T\cong T_1\times T_2\times T_3\times T_4.$$
The product $T_3\times T_4$ is a rectangular truss, where  $T_3$ is left zero truss and $T_4$ is right zero truss. In analogy to the band theory, we show that every rectangular truss is a product of left zero truss and right zero truss. The product $T_1\times T_2$ is a truss constructed on the product $A\times B$ of Boolean rings $A$ and $B$. In particular, for $T_1$, the truss multiplication is given by the Jacobson circle operation on $A$, which corresponds to the sum of sets in the Boolean algebra associated with $A$. The truss on $A$ will be denoted by $\tT(A,\circ)$. For $T_2$, the truss multiplication on $B$ is the Boolean ring multiplication, and we denote such a Boolean truss by $\tT(B)$.

Furthermore, we show when two Boolean trusses are isomorphic. Theorem \ref{thm:iso} let us restrict our consideration to the case when components corresponding to products of Boolean rings are isomorphic. This also simplifies our consideration to commutative Boolean trusses. In Theorem \ref{thm:class1}, we show that two Boolean trusses $$\tT(A,\circ)\times\tT(B)\quad \text{and} \quad\tT(A',\circ)\times \tT(B')$$ constructed from Boolean rings $A,A',B,B'$ are isomorphic if and only if $$A'\cong (1-a)A\oplus bB\quad \text{and} \quad B'\cong aA\oplus (1-b)B$$ for some $a\in A$ and $b\in B$. This lets us characterise commutative Boolean trusses on $A\oplus B$ based on whether either of $A$ or $B$ is principal.

We conclude the paper with two examples. In the first, Example \ref{ex:15}, we count all non-isomorphic Boolean trusses on a finitely generated Abelian group. In the second one, Example \ref{ex:fartoolong}, we study the Boolean ring $\prod_{i\in\NN}F_i$, where $F_i\cong\ZZ_2$, its ideal $P=\bigoplus_{i\in\NN}F_i$, and the maximal ideal $R$ containing $P$. We show that $R$ cannot be decomposed into a direct sum of two non-principal ideals and show that the only commutative Boolean trusses corresponding to $R$ are $\tT(R)$ and $\tT(R,\circ)$. For $P$, we show that there are three non-isomorphic commutative Boolean trusses, $\tT(P),\tT(P,\circ)$, and $\tT(P,\circ)\times \tT(P)$. Finally, we describe all the commutative Boolean trusses on $R\times P$. Let $\tilde{R}$ denote $R$ with adjoined identity as in \cite[Page 3]{DRKIR}. Surprisingly, by the power of a Stone space, there exist $2^{\aleph_0}$ non-isomorphic maximal ideals $R$ containing $P$ in $F$ such that $\tilde{R}\cong F.$

\section{Preliminaries}
\subsection{Bands}
First, let us recall the notion of a band, as bands will form the
underlying structure of the objects of interest. A {\bf band} is a semigroup $(B,\cdot)$ such that $b\cdot b=b$ for all $b\in B.$ That is, a band or ``a band of groups of order one" (See \cite{Clifford1954}) is exactly an idempotent semigroup. The following three examples will be of particular interest to us:

\begin{example}
     Let $L,R$ be sets. Then:
     \begin{enumerate}
         \item $L$ together with the operation $a\cdot b=a$ for all $a,b\in L$ is called a {\bf left zero band}.
         \item $R$ together with the operation $a\cdot b=b$ for all $a,b\in R$ is called a {\bf right zero band}.
         \item A band $B$ is called a {\bf rectangular band} if $aba=a$ for all $a,b\in B.$
     \end{enumerate}
\end{example}

The following lemma shows that every rectangular band is the direct product of a left zero band and a right zero band; see McLean \cite{McLean1954} and  Kimura \cite[Lemma 1]{Kimura1958}.

\begin{lemma}\label{lem:rect}
    A band $B$ is a rectangular band if and only if $B\cong L\times R$ for some left zero band $L$ and right zero band $R.$
\end{lemma}

In fact, every band can be described in terms of a semilattice of rectangular bands, by a
result of McLean \cite{McLean1954} sharpened by Clifford
\cite{Clifford1954}.

\begin{theorem}[\cite{McLean1954} and \cite{Clifford1954}]\label{thm:cliff-McLean}
    Every band is a semilattice of rectangular bands.
\end{theorem}

\subsection{Heaps and Trusses}

In 1920's Baer \cite{Baer29}, Pr\"ufer \cite{Pruf24} and Su\v skevi\v c \cite{Sus37} introduced an algebraic structure called a heap. For our purposes we will consider only Abelian case. An \textbf{Abelian heap} is a set $H$ together with a ternary operation such that for all $a_1,a_2,a_3,a_4,a_5\in H:$
 
\begin{subequations}
\begin{align} [[ a_1,a_2,a_3 ] , a_4,a_5 ]  &= [ a_1,a_2,[ a_3,a_4,a_5 ]  ],\\  [ a_1,a_2,a_2 ] =a_1 &\ \&\  a_2 = [ a_2,a_1,a_1 ]\label{eq:mal}, \\  [ a_1,a_2,a_3]&=[ a_3,a_2,a_1].
    \end{align}
\end{subequations} 
 A {\bf subheap} of an Abelian heap $(H,[-,-,-])$ is a subset $S\subseteq H$ such that $(S,[-,-,-])$ is an Abelian heap and a \textbf{ heap homomorphism} from an Abelian heap $H$ to an Abelian heap $S$ is a function $\varphi:H\to S$ such that for all $a,b,c\in H,$
$$
\varphi([ a,b,c ] )=[ \varphi(a),\varphi(b),\varphi(c) ] .
$$

There is a deep connection between Abelian heaps and Abelian groups. For any fixed $e\in H,$ we define a binary operation $+_e$ in $H$ by
$$
a+_eb=[ a,e,b].
$$
The pair $(H,+_e)$ is an abelian group (\textbf{called a retract of $H$ in $e$}) and $-_eh=[ e,h,e ],$ for any $h\in H.$ Moreover, for any $e,f\in H,$ groups $(H,+_e)$ and $(H,+_f)$ are isomorphic. 

Consider an Abelian group $(A,+)$ and the ternary operation $[ a,b,c]=a-b+c$ for any $a,b,c\in A,$ then the pair $\mathrm{H}(A):=(A,[-,-,-])$ is an Abelian heap. In greater generality, if $S\leq A$ and $e\in A$, then $(e+S,[-,-,-])$ is an Abelian heap.

Conceptually, an Abelian heap is an Abelian group up to the choice of a neutral element; that is, it is an affine analogue of an Abelian group.

Recently, in \cite{Brzezinski2019} an algebraic structure of a {\bf truss} was introduced as an Abelian heap $(T,[ -,-,- ] )$ endowed with a semigroup operation $\cdot:T\times T\to T$ such that for all $a,b,c,d\in T,$
    \begin{equation}
  a\cdot [ b,c,d ]=[  a\cdot b, a\cdot c,  a\cdot d ]  \quad \&\quad [ b,c,d ] \cdot a=[ b\cdot a,c\cdot a,d\cdot a ] .
    \end{equation}
     A {\bf subtruss} of a truss $(T,[-,-,-],\cdot)$ is a subset $S\subseteq T$ such that $(S,[-,-,-],\cdot)$ is a truss and   a \textbf{truss homomorphism} $\varphi$ between two trusses $T$ and $S$ is a heap homomorphism which is also a semigroup homomorphism, that is, 
    $$
    \forall{a,b\in T}\quad  \varphi(ab)=\varphi(a)\varphi(b).
    $$

In particular, trusses generalise rings as follows: If $(T,[-,-,-],\cdot)$ is a truss and there exists $a\in T$ such that for all $t\in T$ $at=ta=a,$ then $(T,+_a,\cdot)$ is a ring, the element $a$ is called an {\bf absorber}. Furthermore, if $a\in T$ is an absorber, then for any surjective homomorphism of trusses $\varphi:T\to S,$ $\varphi(a)$
is the absorber in $S.$ Indeed, for all $s\in S$ there exists $t\in T$ such that $\varphi(t)=s$ and
\begin{equation}
    s\cdot \varphi(a)=\varphi(ta)=\varphi(a)=\varphi(at)=\varphi(a)\cdot s.
\end{equation}
Therefore, $\varphi(a)$ is the absorber in $S.$ Moreover, $\varphi:(T,+_a,\cdot)\to (S,+_{\varphi(a)},\cdot)$ is a ring homomorphism as it preserves the absorber (zero of the ring) and \begin{equation}\label{eq:abs}
    \varphi(c+_ad)=\varphi([c,a,d])=[\varphi(c),\varphi(a),\varphi(d)]=\varphi(c)+_{\varphi(a)}\varphi(d)
\end{equation}
for all $c,d\in T.$
The following trusses examples will be useful in what follows.
\begin{example}\label{ex:jacobtruss}
Let $R$ be an arbitrary associative ring and $\mathrm{H}(R)$ be a heap on a set $R$ with ternary operation $[ a,b,c]=a-b+c$ for all $a,b,c\in R$:
    \begin{itemize}
        \item Then $\mathrm{H}(A)$ with $a\cdot b=ab$ is a truss which we denote by $\mathrm{T}(R)$.
        \item Then $\mathrm{H}(A)$ with $a\circ b=a+b+ab$ is a truss which will be called a \textbf{circle truss} and denoted by $\tT(R,\circ).$
        \item Let $(S,+,\cdot)$ be a ring, $R\triangleleft S$ and $e\in S$ such that $e^2-e\in R.$ Then $(\mathrm{H}(e+R),\cdot)$ is a truss.
    \end{itemize}
    \end{example}

\begin{lemma}\label{lem:isoringcirc}
    Let $R$ and $S$ be rings. Then $\tT(R,\circ)\cong \tT(S,\circ)$ if and only if $R\cong S.$
\end{lemma}
\begin{proof}
    If $f:R\to S$ is an isomorphism of rings, then for all $r,r'\in R$
    $$
f(r\circ r')=f(rr'+r+r')=f(rr')+f(r)+f(r')=f(r)f(r')+f(r)+f(r')=f(r)\circ f(r'),
    $$
and $f:\tT(R,\circ)\to \tT(S,\circ)$ is an isomorphism of trusses. The proof for the converse is analogous, as $rr'=r\circ r'-r-r'$ for all $r,r'\in R$ and $f$ preserves identities which becomes zeros of rings.
\end{proof}

The following lemma can be found in \cite[Proposition 3.12(4)]{ABR}.

\begin{lemma}\label{lem:circmult}
Let $R$ be a ring. Then $\tT(R,\circ)\cong\tT(R)$ if and only if $R$ is unital.
\end{lemma}
\begin{proof}
    If $R$ is unital, then $1$ is an absorber of $\tT(R,\circ)$ and $$[-,1,0]:\tT(R,\circ)\to \tT(R),\quad x\mapsto [x,1,0]$$ is an isomorphism of rings. If $\tT(R,\circ)\cong\tT(R),$ then since $\tT(R,\circ)$ is unital and isomorphism preserves identity, $R$ is unital. 
\end{proof}

\begin{example}\label{ex:lrz}
Let $(A,[-,-,-])$ be an Abelian heap. 
\begin{enumerate}
    \item A \textbf{left zero truss} $\zZ_\lL (A)$ is a heap $A$ with the left zero band operation $a\cdot_l b=a$ for all $a,b\in A$
    \item A \textbf{right zero truss} $\zZ_\rR(A)$ is a heap $A$ with the right zero band operation $a\cdot_r b=b$ for all $a,b\in A.$
\end{enumerate}
\end{example}

The following lemma is introduced without a proof. The proof is a simple exercise.

\begin{lemma}\label{lem:zerogroup}
    Let $H$ and $K$ be two non-empty Abelian heaps. Then the following are equivalence:
    \begin{enumerate}
        \item $H\cong K,$
        \item $\zZ_\lL(H)\cong \zZ_\lL(K),$
        \item $\zZ_\rR(K)\cong \zZ_\rR(H),$
        \item $(H,+_h)\cong (K,+_k)$ for all $h\in H$ and $k\in K.$
    \end{enumerate}
    Moreover, $\zZ_\lL(H)\cong\zZ_\rR(K)$ if and only if $|H|= |K|=1.$ 
\end{lemma}

In \cite{ABR} the authors have shown that with every ring one can associate multiple trusses using maps called double homothetisms. For our purposes, we present a weaker version of \cite[Theorem 3.6 (3)]{ABR}. See also \cite[Definition 3.7 (3)]{ABR}.

\begin{lemma}\label{lem:hom}
    Let $R$ be a commutative ring, and $\sigma\in \mathrm{End}_R(R)$ be an $R$-module endomorphism such that $\sigma^2=\sigma.$ Then $\hH(R^{+})$ together with multiplication
    $$
a\diamond b=ab+\sigma(a)+\sigma(b)
    $$
    is a truss, which we will denote by $T(\sigma,0).$
\end{lemma}

Given a truss, one can always obtain a ring by using \cite[Theorem 4.3]{ABR}, or by employing a weaker version that is more suitable for our purposes:

\begin{lemma}\label{lem:const:bullet}
    Let $(T,[-,-,-],\cdot)$ be a truss and $e\in T.$ Then $\rR(T,e):=(T,+_e,\bullet_e)$ is a ring, where 
    $$
a\bullet_e b:=[[a\cdot b,a\cdot e,e],e\cdot b,e^2]
    $$
    for all $a,b\in T.$
\end{lemma}

In fact, by the \cite[Theorem 4.3]{ABR}, the choice of $e\in T$ in the previous lemma is arbitrary.

\begin{remark}
Let $(T,[-,-,-],\cdot)$ be a truss, then $\rR(T,e)\cong \rR(T,f)$ for all $e,f\in T.$ Therefore, the notation $\rR(T):=\rR(T,e)\cong\rR(T,f)$ will be used.
\end{remark}

\begin{remark}\label{rem:iso}
    If $T$ and $S$ are isomorphic trusses, then $\rR(T)$ and $\rR(S)$ are isomorphic rings.
\end{remark}

\begin{lemma}\label{lem:zero=zero}
Let $(T,[-,-,-],\cdot)$ be a left zero truss or a right zero truss. Then $\rR(T)$ is a zero multiplication ring.
\end{lemma}
\begin{proof}
    Let $(T,[-,-,-],\cdot)$ be a left zero truss and $e\in T.$ Then for all $a,b\in T,$
    $$
a\bullet_e b=[[a\cdot b, a\cdot e,e],e\cdot b,e^2]=[[a,a,e],e,e]=e.
    $$
The third equality follows by \cref{eq:mal}. Therefore, $\rR(T)$ is a zero multiplication ring. The check for a right zero truss is analogous.
\end{proof}

\section{Main Part}
Truss is a generalisation of a ring. There are Boolean rings; is there a ``non-ring" truss in which every element is an idempotent? The answer is yes and they can be noncommutative, see Example \ref{ex:lrz}. Knowing that, during one of the Bia\l ystok University Algebra Seminars second author  stated the following problem:
\begin{quote}
    Describe all trusses $T$ such that $a^2=a$ for all $a\in T.$
\end{quote}

Let us now introduce the main structure of interest:

\begin{definition}
    A \textbf{Boolean truss} is a truss $T$ such that $a^2=a$ for all $a\in T.$
\end{definition}

By Theorem \ref{thm:cliff-McLean} and since every Boolean truss has the underlying structure of a band, the rectangularity will play a crucial role in future considerations. Therefore let us introduce analogous definition of rectangularity for trusses:

\begin{definition}
    A truss $T$ is called {\bf rectangular} if $aba=a$ for all $a,b\in T.$
\end{definition}

\begin{example}
\leavevmode
    \begin{itemize}
        \item Every left zero truss and right zero truss is Boolean and rectangular.
        \item If $R\not =\{0\}$ is a Boolean ring, then $\mathrm{T}(R)$ is a Boolean truss, but it is not rectangular.
        \item If $R\not =\{0\}$ is a Boolean ring, then $\tT(R,\circ)$ is a Boolean truss, but it is not rectangular.
        \item A direct product of Boolean (rectangular) trusses is a Boolean (rectangular) truss.
    \end{itemize}
\end{example}

For the convenience of the reader, we also add an explicit example of a rectangular truss:

\begin{example}\label{ex:diff}
    Consider the Abelian heap $\hH(\mathbb{Z}/30\mathbb{Z})$ acquired from the additive group $(\mathbb{Z}/30\mathbb{Z},+)$ together with multiplication given by
    $$
a\cdot b=6b-5a,
    $$
    for all $a,b\in \mathbb{Z}/30\mathbb{Z}.$ Then $(\mathbb{Z}/30\mathbb{Z},\ll-,-,-\rr,\cdot)$ is a rectangular truss. 
\end{example}

\begin{lemma}
    Let $T$ be a truss and $a\in T.$ Then $C_T(a):=\{x\in T\ |\ xa=ax\}$ is a subtruss of $T.$
\end{lemma}
\begin{proof}
    For any $b,c,d\in C_T(a),$
    $$[b,c,d]a=[ba,ca,da]=[ab,ac,ad]=a[b,c,d].$$
    Therefore, $C_T(a)$ is a subheap. Furthermore, $bca=bac=abc,$ and $C_T(a)$ is a subtruss.
\end{proof}

\begin{proposition}\label{prop:rectbool}
Let $T$ be a rectangular truss. Then $T$ is Boolean and $C_T(a)=\{a\}$ for all $a\in T.$
\end{proposition}
\begin{proof}
Let $b=a^4$ for any $a\in T,$ then 
$$
a=aba=aa^4a=a^3a^3=a^2,
$$
and $T$ is Boolean. Now, for any $x\in C_T(a),$
$$
a=axa=a^2x=ax\quad \&\quad x=xax=ax.
$$
Therefore, $x=a$ for all $x\in C_T(a),$ and $C_T(a)=\{a\}.$
\end{proof}

The following theorem is a consequence of \cite{ABR} and allows us to interpret a truss as a congruence class of a ring, closed under multiplication of the ring.

\begin{theorem}\cite[Theorem 3.6 and Section 5]{ABR}\label{thm:ABR}
    Every Truss $T$ is isomorphic with a truss of the form $(e+R,[ -,-,-],\cdot),$ where $(S,+,\cdot)$ is a ring, $R\triangleleft S,$ 
    $S/R\cong \mathbb{Z},$ $e\in S$, $e^2-e\in R$ and $[ a,b,c]=a-b+c$ for all $a,b,c\in S$. 
\end{theorem}

Now, we are ready to state and prove the main theorem.

\begin{theorem}\label{thm:main}
    A truss $T$ is Boolean if and only if it is isomorphic to a truss $$T_1\times T_2\times T_3\times T_4,$$
    where 
    \begin{enumerate}
        \item $T_1=\tT(A,\circ)$ is a circle truss of a Boolean ring $A,$
        \item $T_2=\mathrm{T}(B)$ for a Boolean ring $B$,
        \item $T_3$ is a left zero truss,
        \item $T_4$ is a right zero truss.
    \end{enumerate}
\end{theorem}
\begin{proof}
It is a direct check that every element of $T_1\times T_2\times T_3\times T_4$ is an idempotent, and therefore $T$ is a Boolean truss.

Now, assume that $T$ is Boolean, then by the Theorem \ref{thm:ABR} there exist rings $S,R$ and $e\in S$ such that $R\triangleleft S$ and $T$ is isomorphic to $(e+R,[ -,-,-],\cdot).$

Since $e=e+0\in e+R,$ $e^2=e$ and $e$ is an idempotent of a ring $S.$ Moreover, for all $x\in R$ $(e+x)^2=e+x,$ thus
\begin{equation}\label{eq:sq}
ex+xe+x^2=x\quad  \text{for\  all }\quad x\in R.
\end{equation}
By substituting $x$ with $x+y$ in \cref{eq:sq} we get that
\begin{equation}\label{eq:sq3}
xy+yx=0\quad  \text{for\  all }\quad x,y\in R.
\end{equation}

Therefore $2x^2=0$ for all $x\in R,$ and together with \cref{eq:sq} we have that

\begin{equation}\label{eq:sq2}
e(2x)+(2x)e=2x\quad  \text{for\  all }\quad x\in R.
\end{equation}

Now, we multiply \cref{eq:sq2} by $e$ from right, and acquire:

\begin{equation}
e(2x)e=0\quad   \text{for\ all}\quad x\in R.
\end{equation}

For any $x,y\in R$ we have that $$2xy=e(2x)y+(2x)ey=e(2x)(ey+ye+y^2)+(2x)e(ey+ye+y^2)=-2(ey)(xe)=0$$ by \cref{eq:sq}, \cref{eq:sq2} and the fact that $R\triangleleft S.$ Thus for every $x,y\in R,$
$$
2xy=0.
 $$
 Therefore, $xy=-xy$ for any $x,y\in R,$ and since \cref{eq:sq3}, we get that $R$ is a commutative ring, i.e.
 $$
xy=yx\quad \text{for\  all}\quad x,y\in R.
 $$
 Let us Denote by $D(S)$ a Dorroh extension of $S.$ Then $R\triangleleft D(S)$ and $1\in D(S),$ then additive group $R^+$ of $R$ has the following Pierce decomposition to direct sum:
 $$
R^+=eRe\oplus (1-e)R(1-e)\oplus(1-e)Re\oplus eR(1-e).
 $$
By the commutativity of $R,$ we get $eRe\triangleleft R$ and $(1-e)R(1-e)\triangleleft R.$ Moreover, for all $x,y\in R$ we have  $(1-e)xey=(1-e)eyx=0\cdot yx=0$ and $yex(1-e)=xye(1-e)=xy\cdot 0=0,$ thus $[eR(1-e)+(1-e)Re]R=\{0\},$ so $eR(1-e)\triangleleft R$ and $(1-e)Re\triangleleft R.$ Therefore the direct sum of subgroups is, in fact, a direct sum of ideals and $[eR(1-e)]^2=[(1-e)Re]^2=\{0\}.$

Now, let us consider the following rings and trusses:
\begin{enumerate}
\item For every $x\in eRe,$ by \cref{eq:sq}, \cref{eq:sq3} and $ex=xe=x,$ we get $x=-x^2=x^2$. Let $T_1$ be a circle truss on a ring $eRe.$

\item For every $x\in (1-e)R(1-e)$ by \cref{eq:sq} and the fact that $ex=xe=0,$ we get $x^2=x.$ Let $T_2$ be a truss $\mathrm{T}((1-e)R(1-e)).$

\item For every $x\in (1-e)Re,$ $ex=0$ and $xe=x.$ Furthermore, for any $x,y\in (1-e)Re,$  $xy=0.$ Therefore $(e+x)(e+y)=e+x.$ Let $T_3$ be a left zero truss on a zero multiplication ring $(1-e)Re$, i.e. the truss multiplication is $x\circ y=x$ for all $x,y\in (1-e)Re.$

\item For every $x\in eR(1-e)$ we have that $ex=x$ and $xe=0.$ Moreover, for any $x,y\in eR(1-e),$  $xy=0,$ so $(e+x)(e+y)=e+y.$ Let $T_4$ be a right zero truss on a zero multiplication ring $eR(1-e),$ i.e. the truss multiplication is $x\circ y=y$ for all $x,y\in eR(1-e).$
\end{enumerate}
Observe that for every $x_1,y_1\in eRe,$ $x_2,y_2\in (1-e)R(1-e),$ $x_3,y_3\in eR(1-e),$ $x_4,y_4\in (1-e)Re,$ we have that
$$
[e+(x_1+x_2+x_3+x_4)]\cdot [e+(y_1+y_2+y_3+y_4)]=e+[(y_1+x_1+x_1y_1)+x_2y_2+x_3+y_4]
$$
It is a standard check to show that a map $e+(x_1+x_2+x_3+x_4)\mapsto (x_1,x_2,x_3,x_4)$ is an isomorphism of a truss $e+R$ on a truss $T_1\times T_2 \times T_3 \times T_4$ with multiplication given by
\begin{equation}\label{eq:main1}
(x_1,x_2,x_3,x_4)\cdot (y_1,y_2,y_3,y_4)=(y_1+x_1+x_1y_1,x_2y_2,x_3,y_4).
\end{equation}
\end{proof}

It is worth noting here that in a decomposition of a Boolean truss $T$ into $T_1,T_2,T_3,T_4$ from the Theorem \ref{thm:main}, the multiplication of $T_1$ is the sum in the Boolean algebra associated to $A$ and the multiplication of $T_2$ is the intersection in the Boolean algebra associated to $B.$ 

Furthermore, Theorem \ref{thm:main} is not only consistent with results from band theory, it is also realising rectangular bands as particular congruence classes in ring theory.

Having the decomposition of the Boolean truss, we can now investigate when two Boolean trusses are isomorphic.

\begin{theorem}\label{thm:iso}
    Let $T$ and $T'$ be Boolean trusses with decompositions as in Theorem \ref{thm:main}. Then $T$ is isomorphic to $T'$ if and only if $T_1\times T_2\cong T'_1\times T'_2,$ $T_3\cong T'_3$ and $T_4\cong T'_4.$
\end{theorem}

\begin{proof}
Let $\Phi: T\to T'$ be a truss isomorphism, $a=(a_1,a_2,a_3,a_4)\in T$ and $\Phi(a)=b=(b_1,b_2,b_3,b_4).$ Then by \eqref{eq:main1} and a simple check, we obtain:

    \begin{align}&C_T(a)=T_1\times T_2\times \{a_3\}\times \{a_4\}\cong T_1\times T_2,\label{eq:iso:1}\\  &C_{T'}(b)=T'_1\times T'_2\times \{b_3\}\times \{b_4\}\cong T'_1\times T'_2,\\  &L_T(a)=\{x\in T\ |\ xa=x\ \land\ ax=a\}=\{a_1\}\times\{a_2\}\times T_3\times \{a_4\}\cong T_3,\\ &L_{T'}(b)=\{x'\in T'\ |\ x'b=x'\ \land\ bx'=b\}=\{b_1\}\times\{b_2\}\times T'_3\times \{b_4\}\cong T'_3, \\ &R_T(a)=\{x\in T\ |\ ax=x\ \land\ xa=a\}=\{a_1\}\times\{a_2\}\times \{a_3\}\times \{T_4\}\cong T_4,\\  &R_{T'}(b)=\{x'\in T'\ |\ bx'=x'\ \land\ x'b=b\}=\{b_1\}\times\{b_2\}\times \{b_3\}\times \{T'_4\}\cong T'_4.
\end{align}

Now, as all cases are proven in a similar manner, we present only the proof for $T_3\cong T'_3.$ Let $x\in L_T(a),$ then 
$$
\Phi(x)b=\Phi(xa)=\Phi(x)\quad \&\quad b\Phi(x)=\Phi(ax)=\Phi(a)=b,
$$
thus $\Phi(x)\in L_{T'}(b),$ and hence $\Phi(L_T(a))\subseteq L_{T'}(b).$ Moreover, by the same calculation for $\Phi^{-1},$ $\Phi^{-1}(L_{T'}(b))\subseteq L_T(a).$ Therefore $\Phi(L_T(a))=L_{T'}(b)$ and $T_3\cong T'_3.$
\end{proof}

The following corollary is a truss version of the Lemma \ref{lem:rect}:
\begin{corollary}\label{cor:rect}
A truss $T$ is rectangular if and only if $T\cong L\times R$ for some left zero truss $L$ and right zero truss $R.$ Moreover, $\rR(T)$ is a zero multiplication ring. 
\end{corollary}
\begin{proof}
   Let $T$ be a rectangular truss. Then by the Proposition \ref{prop:rectbool}, $T$ is Boolean and $C_T(a)=\{a\}$ for all $a\in T.$ Therefore, by the Theorem \ref{thm:main} and \cref{eq:iso:1} from the proof of Theorem \ref{thm:iso}, we get $T\cong T_3\times T_4,$ where $T_3$ is a left zero truss and $T_4$ is a right zero truss. Conversely, it is a simple check that $L\times R$ is a rectangular truss. The ring $\rR(T)$ is a zero multiplication ring. By the the Lemma \ref{lem:zero=zero}, $\rR(L)$ and $\rR(R)$ are zero multiplication rings. By the Remark \ref{rem:iso}, $\rR(T)$ is isomorphic to $\rR(L\times R).$ Clearly, $\rR(L\times R)=\rR(L)\times \rR(R),$ by the construction of the ring $\rR(L\times R)$ given in the Lemma \ref{lem:const:bullet}.
\end{proof}

The following remark is a straightforward consequence of the Corollary \ref{cor:rect} and the Lemma \ref{lem:zerogroup}.

\begin{remark}
    Two rectangular trusses $L\times R$ and $L'\times R'$ are isomorphic if and only if $L\cong L'$ and $R\cong R'.$ Moreover, $L\cong L'$ $(R\cong R')$ if and only if $(L,+_e)\cong (L',+_{e'})$ $((R,+_e)\cong (R',+_{e'}))$ for all $e\in L\ (e\in R)$ and $e'\in L'\ (e'\in R').$
\end{remark}

By Theorem \ref{thm:iso}, the characterisation of all Boolean trusses simplifies to describing all Boolean trusses of the form $T_1 \times T_2.$ Moreover, by the proof of Theorem \ref{thm:main}, $T_1 \times T_2$ is a product of trusses constructed on some Boolean ring $P$ which is a direct sum of its two ideals $A$ and $B$. The correspondence is as follows, $T$ is a {\bf commutative Boolean truss} if and only if there exist ring $P$ and ideals $A$ and $B$ of $P$ such that $P=A\oplus B$ and $\tT(A,\circ)\times\tT(B)=T.$ For convenience, we denote
$$
\bB(A,B):=\tT(A,\circ)\times \tT(B).
$$
Therefore, let us further investigate the various cases that arise for an arbitrary Boolean ring $P$.

\begin{theorem}\label{thm:class1}
    Let $A,A,B,B'$ be ideals of a Boolean ring $P$ such that $P=A\oplus  B=A'\oplus B'.$ Then the following are equivalent:
    \begin{enumerate}
        \item $\bB(A,B)\cong \bB(A',B')$
        \item there exist $a\in A$ and $b\in B$ such that $A'\cong(1-a)A\oplus bB$ and $B'\cong aA\oplus (1-b)B$
    \end{enumerate}
\end{theorem}
\begin{proof}
$(1)\implies(2)$ Since $P=A\oplus B,$ by the \cref{eq:main1} from the proof of Theorem \ref{thm:iso}, the multiplication of $\bB(A,B)$ is given by
$$
(a_1+b_1)\cdot(a_2+b_2)=(a_1a_2+a_1+a_2)+b_1b_2
$$
for all $a_1,a_2\in A$ and $b_1,b_2\in B.$
Both sets $A$ and $B$ can be described as follows:
$$
A=\{x\in T\ |\  x\cdot  0=0\cdot x=x\}\quad \&\quad B=\{x\in T\ |\ x\cdot 0=0\cdot x=0\}.
$$
Indeed, 
$$
0\cdot(a_1+b_1)=(a_1+b_1)\cdot 0=a_1
$$
for all $a_1+b_1\in P.$ Furthermore, let $\sigma:P\to P$ be given by $\sigma(a_1+b_1):=a_1$ for all $a_1\in A$ and $b_1\in B.$ Then $\sigma\in \mathrm{End}_P(P),$ and by Lemma \ref{lem:hom}, there exists a truss $T(\sigma,0)$ on $P.$ In fact $T=T(\sigma,0)$ as
\begin{equation}\label{eq:alltruss}
(a_1+b_1)\diamond (a_2+b_2)=a_1a_2 +b_1b_2+ \sigma(a_1+b_1)+\sigma(a_2+b_2)=a_1a_2 +b_1b_2+a_1+a_2
\end{equation}
for all $a_1,a_2\in A$ and $b_1,b_2\in B.$ Let $\psi: P\to P$ be a bijection given by $\psi(x):=x+a+b$ for a fixed $a\in A$ and $b\in B.$ Then $\psi$ is an isomorphism given in \cite[Lemma 3.9]{ABR} from $T(\sigma,0)$ to $T(\sigma',0),$ where $\sigma'=\sigma-l_{a+b}$ and $l_{a+b}(x):=(a+b)x.$

For any $x\in A$ and $y\in B,$ we have 
$$
\sigma'(x+y)=\sigma(x+y)+(x+y)(a+b)=x+ax+by,
$$
that is, $$A_1:=\im(\sigma')=(1-a)A\oplus bB.$$ Furthermore, $x+ax+by=0$ if and only if $(1+a)x=0$ and $by=0.$ Therefore, $x\in aA$ and  $y\in (1+b)B.$ Moreover, 
$$
B_1:=\ker(\sigma')=aA\oplus (1-b)B.
$$

Finally, if $\varphi: \bB(A,B)\to \bB(A',B')$ is an isomorphism of trusses, then there exist $a\in A$ and $b\in B$ such that $\varphi(a+b)=0.$ Setting $\varepsilon:= \varphi\circ \psi^{-1},$ we obtain the isomorphism between $T(\sigma',0)$ and $\bB(A',B')$ such that $\varepsilon(0)=0$ as $\psi^{-1}(0)=a+b.$ Therefore, by \cref{eq:abs} and the discussion preceding it, $\varepsilon:P\to P'$ is a ring isomorphism, and 
$$
\varepsilon(A_1)=A'\quad 
 \text{and}\quad \varepsilon(B_1)=B'.
 $$

$(2)\implies (1)$ It is easy to check that both $\tT(-)$ and $\tT(-,\circ)$ preserve finite products. Since $aA$ and $bB$ are unital, we have $\tT(aA,\circ)\cong \tT(aA)$ and $\tT(bB,\circ)\cong\tT(bB)$ by Lemma \ref{lem:circmult}. Furthermore, by the Lemma \ref{lem:isoringcirc}, for any two rings $R$ and $S,$  $\tT(R,\circ)\cong\tT(S,\circ)$ if and only if $R\cong S.$ Therefore,
$$
\begin{aligned}
\bB(A,B)&=\tT(A,\circ)\times \tT(B)\cong \tT(aA\oplus (1-a)A, \circ)\times \tT((1-b)B\oplus bB)\\ &\cong \tT((1-a)A,\circ)\times \tT(aA, \circ)\times \tT((1-b)B)\times \tT(bB)\\ &\cong \tT((1-a)A\oplus bB, \circ)\times \tT(aA\oplus (1-b)B)\cong \tT(A',\circ)\times \tT(B')=\bB(A',B').
\end{aligned}$$
\end{proof}

\begin{corollary}
    Let $P$ be a unital Boolean ring, then $\bB(A,B)\cong \tT(P)$ for all ideals $A$ and $B$ of $P$ such that $P=A\oplus B.$
\end{corollary}
\begin{proof}
    Let $a_1+b_1=1\in P$ for some $a_1\in A$ and $b_1\in B.$ By taking, in Theorem \ref{thm:class1}, $a=a_1$ and $b=0,$ we get that $A'\cong \{0\}$ and $B'\cong P.$ Therefore, $$\bB(A,B)\cong \bB(A',B')\cong\bB(\{0\},P)=\tT(P).$$
\end{proof}

\begin{corollary}\label{cor:principal}
       Let $P=A\oplus B$ be a non-unital Boolean ring such that either $A$ or $B$ is a principal ideal of $P$, then $\bB(A,B)\cong \tT(P)$  or $\bB(A,B)\cong \tT(P,\circ).$  
\end{corollary}
\begin{proof}
    If $B$ is principal, then $B$ is unital with identity $1_B.$ Taking in Theorem \ref{thm:class1}, $a=0$ and $b=1_B,$ we have
    $$
    \bB(A,B)\cong \bB(P,\{0\})=\tT(P,\circ).
    $$
    Analogously, if $A$ is principal, it is unital with identity $1_A.$ By taking $a=1_A$ and $b=0,$ we have 
    $$
    \bB(A,B)\cong \bB(\{0\},P)=\tT(P).
    $$
\end{proof}

\begin{corollary}
Let $P=A\oplus B$ for some non-principal ideals $A$ and $B$ of $P,$ then $\bB(A,B)$ is not isomorphic to either $\tT(P)$ or $\tT(P,\circ).$
\end{corollary}
\begin{proof}
    If $\bB(A,B)\cong \tT(P)=\bB(\{0\},P),$ then by Theorem \ref{thm:class1} exist $p\in P$ and $0\in \{0\},$ such that 
    $A\cong \{0\}\oplus pP=pP.$ Therefore, $A$ is principal and we get a contradiction. Thus $\bB(A,B)$ is not isomorphic to $\tT(P).$ Similarly, if $\bB(A,B)\cong \tT(P,\circ)=\bB(P,\{0\}),$ then $B\cong pP\oplus \{0\}=pP.$ Thus $B$ is principal, we get a contradiction, and  $\bB(A,B)$ is not isomorphic $\tT(P,\circ).$
\end{proof}

\begin{example}\label{ex:15}
    Let $A$ be a finitely generated Abelian group. Then
    $$
A\cong (\ZZ_2)^{\alpha_0}\times (\ZZ_{q_1})^{\alpha_1}\times\ldots\times  (\ZZ_{q_{n-1}})^{\alpha_{n-1}}\times \ZZ^{\alpha_n},
    $$
where $n\in \NN,$ $q_0:=2,$ $q_i$ are distinct powers of prime numbers and $\alpha_i\in \NN\cup\{0\}$ for $i\in\{0,1,\ldots,n\}.$ There are exactly 
$$
\binom{\alpha_0+2}{2}\prod\limits_{i=1}^{n}(\alpha_i+1)
$$
Boolean trusses on $A.$
Indeed, let us consider $A$ and its decomposition. Observe that on an Abelian group $Z:=(\ZZ_{q_1})^{\alpha_1}\times\ldots \times \ZZ^{\alpha_n},$ we can construct a rectangular truss. All such trusses are given by a product of left and right zero trusses. Therefore, we need to count all non-isomorphic decomposition of $Z=Z_1\times Z_2.$ This corresponds to finding all the possible choices of $(k_1,k_2,\ldots,k_n)$ for $k_i\in \{0,\ldots,\alpha_i\}$ and all $i\in\{1,\ldots,n\}.$ There are exactly $\prod\limits_{i=1}^{n}(\alpha_i+1)$ such choices. Now, on the part $(\ZZ_2)^{\alpha_0},$ we need to consider a decomposition of it into $3$ possible subgroups corresponding to Boolean truss, left zero truss and right zero truss. Observe that since group is finitely generated, Boolean ring is finite and unital, and thus circle truss on it is isomorphic to a ring truss. This corresponds to the following decomposition of $\alpha_0=n_1+n_2+n_3$ for some non-negative integers $n_1,n_2$ and $n_3.$ One can easily check that there are $$\binom{\alpha_0+2}{2}$$
such choices. Therefore, there are total of 
$$
\binom{\alpha_0+2}{2}\prod\limits_{i=1}^{n}(\alpha_i+1)
$$
non-isomorphic Boolean trusses described on $A.$
\end{example}
Before we approach our last example, let us recall a basic property of Boolean rings, see \cite[Theorem 2]{AK08}.
\begin{remark}\label{rem:boolprod}
    Let $R$ and $S$ be Boolean rings. Then $K\triangleleft R\times S$ if and only if $K=I\times J$ for $I\triangleleft R$ and $J\triangleleft S.$
\end{remark}

We conclude the paper with the following example

\begin{example}\label{ex:fartoolong}
    Let $F=\prod\limits_{i\in\NN}^{}F_i$ where $F_{i}\cong \mathbb{Z}_2,$ then $$P=\bigoplus\limits_{i\in \NN}^{}F_i$$ is an essential ideal of $F$ and by the Zorn's lemma there exists a maximal ideal $R$ of $F$ such that $P\subset R$ and $F/R\cong \ZZ_2$ as $F$ is Boolean.  Clearly, $R$ has no identity. If $R$ had an identity, then $R$ would be direct summand in $F,$ but $R$ is essential in $F.$ Let us consider the following examples:
    \begin{enumerate}
        \item\label{ex:1} Let $I,J\lhd R$ be such that $R=I\oplus J$ and $e_i\in F_i$ be an identity in $F_i$ for $i\in \NN.$ Then for any $n\in \NN,$ $e_n=i+j$ for some $i\in I$ and $j\in J.$ Indeed, 
    $$e_ni=i^2=i \quad \text{and}\quad e_nj=j^2=j.$$ Moreover, since $e_n\in F_n$ is an identity, we get $$e_n=i\quad \text{or}\quad e_n=j.$$
    Let $$X=\{n\in\NN\  |\   e_n\in I\}\quad \text{and} \quad Y=\{n\in\NN\ |\ e_n\in J \},$$ then $X\sqcup Y=\NN$ and $X,Y$ are non-empty as $P$ is essential ideal. Therefore, for all $i\in I$ and $j\in J,$ $$supp(i)\subseteq X\quad  \text{and} \quad supp(j)\subseteq Y.$$ Since $R$ is a maximal ideal of the Boolean ring $F,$ we have \begin{equation}\label{eq:aleph}F/R\cong \ZZ_2\quad  \text{and}  \quad F=R\sqcup(1+R).\end{equation} Thus, if $a\in F,$ then $a\in R$ or $1-a\in R.$ If $a=(a_n),$ $a_n=e_n$ for $n\in X$ and $a_n=0$ for $x\in Y.$  If $a\in R,$ then $a\in I,$ and $I=(a)=aR.$ Similarly, if $a\in 1+R,$ then $J=(1+a)=(1+a)R.$ In either case $I$ or $J$ is a principal ideal. That is, $R$ cannot be decomposed into two non-principal ideals. Furthermore, by Equation \eqref{eq:aleph}, we observe that the ring is infinite, $|R|=2^{\aleph_0}.$

If $a\in R,$ then $1-a\not\in R$ and $R=aR\oplus (1-a)R.$ Moreover $(1-a)R$ cannot be decomposed as a direct sum of two non-principal ideals. Indeed, if there would exist non-principal $I,J\triangleleft R,$ then 
\begin{equation}\label{eq:ideal}
   R= aR\oplus (I\oplus J)= (aR\oplus I)\oplus J. 
\end{equation}
Observe that $aR\oplus I$ and $J$ are non-principal, thus we get a contradiction with the fact that $R$ cannot be decomposed into direct sum of two non-principal ideals.

If $a\in R\setminus P,$

\begin{equation}\label{eq:afr}
aR=aF\cong F\quad \text{and} \quad (1-a)F\cong F 
\end{equation} because
\begin{equation}
|supp(1-a)|=\aleph_0=|supp(a)|.
\end{equation}
Furthermore, 
\begin{equation}\label{eq:decompositon}
    F=aR\oplus (1-a)F\quad \text{and}\quad R=aR\oplus (1-a)R.
\end{equation} 
Finally, 
\begin{equation}\label{eq:maxxx}
    (1-a)F/(1-a)R\cong \ZZ_2,
\end{equation}
 and by $\eqref{eq:maxxx},$ we get that \begin{equation}\label{eq:maxxxx}
     |(1-a)R|=2^{\aleph_0}.
 \end{equation}

   \item\label{ex:2} Let us consider a non-zero ideal $K\triangleleft P.$ Then for every $a=(a_k)_{k\in \NN},$ we have 
    $$
    a_ke_k=ae_k\in K,
    $$
Therefore, every ideal can be written as $I(M)=\bigoplus\limits_{m\in M}\{0,e_m\}$ for some subset $M\subseteq\mathbb{N}.$ Moreover, $I(\mathbb{N})=P$ and for every subset $M\subseteq\mathbb{N}$,
\[
I(M)\oplus I(\mathbb{N}\setminus M)=P.
\]

The only possible cases of commutative Boolean trusses on $P$ are the following:

\begin{enumerate}
    \item\label{N1}
    Both $M$ and $\mathbb{N}\setminus M$ are infinite. Then
    $I(M)\cong I(\NN\setminus M)\cong I(\NN)\cong P,$ and
    $$
    \bB(I(M), I(\NN\setminus M))\cong \bB(P,P)\cong \tT(P,\circ)\times\tT(P).
    $$
    \item\label{N2}
    If $M$ is a non-empty finite set, then $I(M)=aP,$ 
    $a=\sum\limits_{k\in M}e_k,$ $I(M)$ is principal, and hence
    $$
\bB(I(M),I(\NN\setminus M))\cong \bB(aP,I(\NN\setminus M))\cong \tT(P),
    $$
    by the proof of Corollary \ref{cor:principal}.
    \item\label{N3}
    Similarly, if $\mathbb{N}\setminus M$ is a non-empty finite set, then $I(\mathbb{N}\setminus M)=bP,$ $b=\sum\limits_{k\in M}e_k,$ $I(\mathbb{N}\setminus M)$ is principal, and hence 
    \[
    \bB(I(M),I(\NN\setminus M))\cong \bB(I(M),bP)\cong \tT(P,\circ),
    \]
    by the proof of Corollary \ref{cor:principal}.
\end{enumerate}

\item By the \eqref{ex:1}, on $R$ there are two non-isomorphic commutative Boolean trusses $$\tT(R)\quad \text{ and }\quad \tT(R,\circ).$$ By the \eqref{ex:2}, on $P$ there are three non-isomorphic commutative Boolean trusses $$\tT(P),\quad \tT(P,\circ)\quad \text{and} \quad \tT(P,\circ)\times \tT(P).$$
Therefore, by the Remark \ref{rem:boolprod}, there are at most six commutative Boolean trusses on $P\times R:$
\begin{enumerate}
\item 
$\tT(P\times R)=\tT(P)\times \tT(R),$
\item 
$\tT(P\times R,\circ)\cong \tT(P,\circ)\times \tT(R,\circ),$
\item 
$\bB(R,P)=\tT(R,\circ)\times \tT(P),$
\item 
$\bB(P,R)=\tT(P,\circ)\times \tT(R),$
\item 
$\bB(P,R\times P)=\tT(P,\circ)\times \tT(R\times P),$
\item 
$\bB(R\times P,P)=\tT(R\times P,\circ)\times \tT(P).$
\end{enumerate}
Since $P$ and $R$ are not unital, $\tT(P\times R)$ is the only truss with absorber. Similarly $\tT(P\times R,\circ )$ is the only unital truss. Therefore, both are non-isomorphic with others.

Proofs of (c) and (d), (c) and (e), (d) and (f) are analogous and follows by the same reasoning, so we will only prove that (c) and (d) are not isomorphic. Let us assume that
$$\bB(R,P)\cong\bB(P,R),$$
then by the Theorem \ref{thm:class1},
$$P\cong(1-a)R\oplus bP\quad  \text{and} \quad R\cong aR\oplus (1-b)P$$
for some $a\in R$ and $b\in P.$ If $a\in P,$ then 
$$|R|=2^{\aleph_0},\quad aR\oplus (1-b)P\cong aP\oplus (1-b)P\quad  \text{and}\quad |aP\oplus (1-b)P|=\aleph_0.$$ Therefore, we get a contradiction with our assumption.
If $a\in R\setminus P,$ then by \eqref{eq:maxxxx} $$|(1-a)R|=2^{\aleph_0}\quad \text{and} \quad |P|=\aleph_0.$$ Again, this leads to a contradiction. That is (c) and (d) are not isomorphic.

Similarly, proofs of (c) and (f), (d) and (e) are follows by the same arguments, so we will only show that (c) and (f) are not isomorphic. Let us assume that $$\bB(R,P)\cong \bB(R\times P,P),$$
then by the Theorem \ref{thm:class1},
$$R\times P\cong(1-a)R\oplus  bP\quad  \text{and} \quad P\cong aR\oplus (1-b)P$$
for some $a\in R$ and $b\in P$. By the Remark \ref{rem:boolprod}, there exist $R_1,R_2\triangleleft R$ and $P_1,P_2\triangleleft P$ such that 
\begin{equation}\label{ex:1-aR}
    (1-a)R\cong R_1\times P_1,\quad  bP\cong R_2\times P_2
\end{equation}
and 
\begin{equation}
   R_1\oplus R_2=R,\quad P_1\oplus P_2=P. 
\end{equation}
By \eqref{ex:1-aR}, $R_2$ and $P_2$ are principal. Similarly, by \eqref{ex:1-aR}, $R_1$ and $P_1$ are non-principal as $R$ and $P$ are not unital. Therefore, by the discussion after \eqref{eq:ideal}, we get a contradiction
 with the fact that $(1-a)R$ cannot be decomposed into two non-principal ideals. That is (c) and (f) are not isomorphic.

Finally, for (e) and (f) let us assume that 
$$\bB(R\times P,P)\cong \bB(P,R\times P)$$
then
$$P\cong(1-a)(R\times P)\oplus bP\quad  \text{and} \quad R\times P\cong a(R\times P)\oplus (1-b)P$$
for some $a\in R\times P$ and $b\in P.$ If $a=(a_1,a_2)\in P\times P,$ then $$|R\times P|=2^{\aleph_0},\quad a(R\times P)=a_1R\times a_2P=a_1P\times a_2P\quad \text{and}\quad |a(R\times P)\oplus (1-b)P|=\aleph_0.$$
Therefore, we get a contradiction with our assumption. If $a=(a_1,a_2)\in R\times P\setminus P\times P,$ then by \eqref{eq:maxxxx} 
$$
(1-a)(R\times P)=(1-a_1)R\times (1-a_2)P\quad \text{and}\quad |(1-a)(R\times P)\oplus bP|=2^{\aleph_0}.
$$
Therefore, since $|P|=\aleph_0,$ we get a contradiction with assumed isomorphism. That is, (e) and (f) are not isomorphic.
\end{enumerate}
\end{example}

\begin{remark}\label{rem:2k}
    Under the notation of the Example \ref{ex:fartoolong}, let $R$ be a maximal ideal of $F$ containing $P.$ It is well known that there exist $2^{2^{\aleph_0}}$ different  maximal ideals $R$, see the power of Stone Space in \cite[Page 45]{Sikorski}. Since every automorphism of $R$ is given by a permutation of $\NN,$ there are at most $2^{\aleph_0}$ such automorphisms. Therefore, there are infinitely many non-isomorphic maximal ideals $R$ containing $P.$
\end{remark}

\begin{remark}
    Under the notation of the Example \ref{ex:fartoolong}, by the \eqref{eq:aleph}, the ring $F$ is isomorphic to $R$ with adjoined identity denoted by $\tilde{R},$ see \cite[Section 1.1 and Exercise 12 to chapter 1]{DRKIR}. Thus, by Remark \ref{rem:2k} there exist $2^{\aleph_0}$ non-isomorphic rings $R$ which have the same ring $\tilde{R}.$ This means that \cite[Exercise 12 (c)]{DRKIR} has a false statement.
\end{remark}

\section*{Acknowledgments}
The authors would like to thank Karol Pryszczepko for his help with Example 3.18 and Remark 3.19.

\end{document}